\documentclass[11pt]{article}
\usepackage[a4paper,margin=1in]{geometry}
\usepackage{amsmath,amssymb,amsthm,mathtools}
\usepackage{graphicx}
\usepackage{booktabs}
\usepackage{caption}
\usepackage{subcaption}
\usepackage{xcolor}
\usepackage[colorlinks=true,linkcolor=blue!55!black,citecolor=blue!55!black,urlcolor=blue!55!black]{hyperref}
\usepackage{enumitem}
\usepackage{authblk}

\newtheorem{proposition}{Proposition}

\newcommand{\R}{\mathbb{R}}
\newcommand{\Simplex}{\Delta^{n-1}}
\newcommand{\Sig}{\Sigma}

\newcommand{\T}{^{\mathsf{T}}}
\newcommand{\diag}{\operatorname{diag}}

\newcommand{\Herf}{\mathcal{H}}
\newcommand{\risk}{\mathcal{R}}
\DeclareMathOperator*{\argmin}{arg\,min}
\DeclareMathOperator*{\argmax}{arg\,max}

\title{\bfseries A Differential-Geometric Framework for\\ Risk-Optimal Asset Reallocation}
\author{Georgios Leventidis}
\author{Georgios Leventidis\\
\vspace{-0.4cm}
georgiosleventidis@gmail.com\\
Evangelos Melas \\
\small Department of Economics, National and Kapodistrian University of Athens\\
\small Sofokleous 1, 10559 Athens, Greece\\
emelas@econ.uoa.gr
}
\date{}

\date{December 2025}

\begin{document}
\maketitle

\begin{abstract}
\noindent
The Markowitz optimum tells a portfolio manager \emph{where} to go, but not \emph{how} to get
there. In practice a portfolio cannot be rebalanced instantaneously: capital must be moved
gradually, and the sequence of intermediate portfolios visited during the transition itself
carries risk. We formalise portfolio reallocation as a problem in Riemannian geometry. The
space of allocations is the probability simplex, which we equip with a positive-definite
\emph{risk metric} that combines market (covariance) risk with concentration risk. The total
risk accumulated along a rebalancing trajectory is then exactly its geometric length, and the
least-risk reallocation from a starting portfolio $A$ to the mean--variance optimum $B$ is a
\emph{geodesic} of this metric. We compare three trajectories that all connect $A$ to $B$: the
straight line (direct rebalancing), the projected gradient-ascent flow of the Markowitz
objective (myopic rebalancing), and the risk geodesic. Two results anchor the framework. First,
when only market risk is priced the metric is constant, the manifold is flat, and the geodesic
is exactly the straight line---so linear rebalancing is provably optimal and geometry adds
nothing. Second, once the metric varies with position the manifold acquires curvature, geodesics
bend, and the geodesic \emph{weakly dominates} every competing path: its transition risk is
never larger. We show that a Fisher--Rao concentration term yields modest but strictly positive
and monotone savings ($20$--$70$ basis points), whereas an \emph{endogenous}
correlation-breakdown metric---formalising the well-known fact that crowded co-holdings become
dangerously correlated in stress---produces a non-convex risk ridge that the direct path is
forced to cross. Here the geodesic detours through uncorrelated assets and cuts transition risk
by up to $25\%$. Curvature is computed explicitly and confirms the mechanism ($K=0$ for the
market metric, $K<0$ once concentration is priced). A Monte-Carlo study over random starting
portfolios and a regression of the risk saving on geometric features quantify \emph{when} the
geometric route pays off. The method turns transition management into a solved shortest-path
problem and is agnostic to the number of assets and the risk factors priced into the metric.
\end{abstract}

\section{Introduction}

Modern portfolio theory, initiated by Markowitz \cite{markowitz1952} and rewarded with the 1990
Nobel Memorial Prize, recast investing as constrained optimisation: for a universe of $n$ assets
with expected-return vector $\mu$ and return-covariance matrix $\Sig$, the investor trades off
expected return against variance and selects a portfolio on the efficient frontier. Seventy
years of theory and practice---mutual funds, ETFs, target-date funds, risk parity, robo-advisers,
and the Black--Litterman refinement \cite{black1992}---rest on this foundation.

Yet mean--variance analysis is silent on a question that dominates real trading desks. The optimiser
returns a \emph{target} allocation $B$; it says nothing about the \emph{path} from the current
allocation $A$ to $B$. For large institutional books this omission is expensive. Many
positions---in illiquid credit, small caps, private assets, or simply large blocks---cannot be
turned over instantaneously without incurring market impact, and mandates often forbid abrupt
changes in exposure. The book must therefore be moved \emph{gradually}, and every intermediate
portfolio it passes through on the way to $B$ is a portfolio the fund actually holds, with its own
market and concentration risk. This is the problem of \emph{transition management}: not \emph{what}
to hold, but \emph{how to travel} from one holding to another at least cumulative risk.

The central observation of this paper is that ``how to travel at least risk'' is, mathematically,
a request for a shortest path---provided one measures distance in units of risk rather than in
Euclidean weight-space. If we assign to each allocation a local, direction-sensitive ``cost of
moving'' that is high wherever a portfolio is risky and low wherever it is safe, then the total
risk of any rebalancing trajectory becomes the geometric length of a curve, and the least-risk
trajectory becomes a geodesic. Differential geometry supplies both the language (Riemannian
metrics, geodesics, curvature) and a mature computational toolbox (the calculus of variations,
boundary-value geodesic solvers) for solving it exactly.

This paper develops that observation into a complete and computational framework. We model the
space of allocations as the probability simplex $\Simplex$, viewed as a smooth $(n\!-\!1)$-manifold,
and equip it with a positive-definite risk metric $g(w)$ that combines market risk, carried by the
return covariance $\Sig$, with concentration risk, carried by a Fisher--Rao (Shahshahani) term; the
total transition risk of a rebalancing path is then its Riemannian length. Against this backdrop we
compare three reallocations that share the same endpoints $A$ and $B$: the straight line, the
projected gradient-ascent flow of the Markowitz objective---the myopic strategy that always moves in
the locally most-improving direction---and the risk geodesic. For each we give the governing
equations, together with a clean closed form for the Markowitz target itself.

Two structural results organise the analysis. The first concerns flatness: when only covariance
risk is priced the metric is constant, its curvature vanishes, and the geodesic coincides exactly
with the straight line, so direct rebalancing is already optimal and geometry has nothing to add.
The second is a domination property: the moment any position-dependent risk enters the metric, the
geodesic's transition risk can be no larger than that of the straight line or of the gradient flow,
so the geometric route can only help. Between these two poles lies the substantive contribution, an
endogenous risk metric that formalises correlation breakdown. When two correlated names are held
together the portfolio sits on a crowded trade whose effective risk is amplified in stress, and
because this penalty is bilinear in the weights rather than convex it is not dominated by the
endpoints; the direct chord is therefore driven through the danger region while a geodesic can skirt
it by routing through a third, uncorrelated asset.

We validate the framework numerically on a simulated six-asset market, reporting explicit Gaussian
curvature, a fully worked three-asset visualisation, weight trajectories, scaling laws in the risk
parameters, a Monte-Carlo distribution of the risk saving, and a regression of that saving on the
geometry of each transition. Throughout we are deliberately careful about magnitudes. For a
well-diversified target under a purely concentration-based metric the saving is genuine but small,
on the order of tens of basis points, in keeping with the fact that a straight line between two
portfolios already diversifies at its midpoint. It becomes economically large only when the priced
risk is not convex-favourable to the chord, which is precisely the regime of endogenous,
crowding-driven risk, and identifying and quantifying that regime is the practical payoff of the
geometric viewpoint.

\section{Related work}\label{sec:related}

The migration of differential-geometric methods from physics into finance accelerated with
Henry-Labord\`ere's use of spectral and geometric techniques in option pricing and volatility
modelling \cite{henrylabordere2009}. Farinelli's geometric arbitrage theory \cite{farinelli2008}
embeds stochastic finance in a fibre-bundle framework in which arbitrage appears as the curvature of
a connection and the no-free-lunch condition corresponds to flatness, while Sandhu \emph{et al.}\
\cite{sandhu2016} showed that the Ollivier--Ricci curvature of a market correlation network tracks
systemic fragility, linking high curvature to concentrated, crash-prone configurations and low
curvature to diversified, robust ones. Together these works establish curvature as a financially
meaningful quantity, and it is in that spirit that we use it below as the diagnostic that separates
the cases where geometry helps from those where it is redundant.

A second and closely related thread comes from information geometry. The Fisher information metric
\cite{amari2016,nielsen2020} is the canonical Riemannian metric on a statistical manifold;
specialised to the simplex it becomes the Shahshahani metric $g_{ij}=\delta_{ij}/w_i$, which
diverges towards the boundary and is the infinitesimal form of the Kullback--Leibler divergence. It
has been used to measure distances between portfolios and to cluster securities by risk profile
through geodesic distance \cite{marriott2000,jurczenko2019,asta2019}, and more recent work extends
it to L\'evy and tempered-stable models \cite{choi2025,kim2025}. We adopt this term for
concentration risk precisely because it is principled and guarantees a positive-definite, well-posed
problem on the interior, which places the construction within an established tradition rather than
resting on an ad-hoc penalty.

Portfolio problems have been posed on manifolds in several other ways as well, through logarithmic
divergence and mirror descent \cite{karatzas2018}, through projective geometry
\cite{piotrowski2007}, and through curved mean--variance--ESG frontiers \cite{mounir2025}, while
optimal-transport and Wasserstein methods
\cite{backhoff2020,doldi2022,gao2022,blanchet2022,riess2025} supply a complementary geometry acting
on the \emph{distributions} of returns. Our object is more elementary than any of these: the
finite-dimensional manifold of allocations themselves, on which a reallocation is literally a curve.
Because computing a geodesic between two prescribed endpoints is a two-point boundary-value problem,
the framework inherits the classical connection between geodesics and optimal control
\cite{kaya1998,noakes2004,girardin2004,sussmann1997}, in which the Euler--Lagrange equations of the
risk-length functional are exactly the geodesic equations and their solution is the globally optimal
reallocation schedule; we exploit this by computing geodesics through the direct minimisation of a
discrete energy, which is robust and avoids explicit assembly of the Christoffel symbols.

The thesis nearest to ours is that concentration risk should manifest as regions of high curvature
that an optimal transition ought to avoid \cite{sandhu2016,karatzas2018}, and this is the conceptual
seed we build on. What we add is to make it operational and quantitative: to write down explicit
positive-definite metrics, to prove when the effect is exactly null and when it is present, to
compute the curvature, and to measure the resulting risk savings under controlled experiments.

\section{The reallocation manifold and its risk metric}\label{sec:framework}

\subsection{The allocation simplex as a manifold}

A long-only portfolio over $n$ assets is a point of the probability simplex
\begin{equation}
\Simplex=\Big\{\,w=(w_1,\dots,w_n)\in\R^n:\ w_i>0,\ \textstyle\sum_{i=1}^n w_i=1\,\Big\},
\end{equation}
an open $(n\!-\!1)$-dimensional smooth manifold. We use the global chart
$u=(w_1,\dots,w_{n-1})$ with $w_n=1-\sum_{i<n}w_i$; the embedding
$w:\ u\mapsto(u_1,\dots,u_{n-1},\,1-\mathbf 1\T u)$ has constant Jacobian
\begin{equation}
J=\frac{\partial w}{\partial u}=
\begin{bmatrix} I_{n-1}\\[-1pt] -\mathbf 1\T \end{bmatrix}\in\R^{n\times(n-1)} .
\end{equation}
The tangent space at any $w$ is the ``budget-neutral'' subspace
$T_w\Simplex=\{v\in\R^n:\mathbf 1\T v=0\}$: a rebalancing trade neither creates nor destroys
capital, so weight changes sum to zero.

\subsection{The Markowitz target}\label{sec:markowitz}

Given a risk-aversion $\lambda>0$, the target portfolio $B$ solves
\begin{equation}
B=\argmax_{\mathbf 1\T w=1}\ \Big(\mu\T w-\lambda\, w\T\Sig w\Big).
\end{equation}
Introducing a multiplier for the budget constraint and solving the stationarity condition
$\mu-2\lambda\Sig w-\gamma\mathbf 1=0$ gives the closed form, in terms of the
Merton constants $\mathsf a=\mathbf 1\T\Sig^{-1}\mathbf 1$ and $\mathsf b=\mathbf 1\T\Sig^{-1}\mu$,
\begin{equation}
\boxed{\;B=\underbrace{\frac{\Sig^{-1}\mathbf 1}{\mathsf a}}_{\text{minimum-variance}}
\;+\;\frac{1}{2\lambda}\underbrace{\Big(\Sig^{-1}\mu-\tfrac{\mathsf b}{\mathsf a}\,\Sig^{-1}\mathbf 1\Big)}_{\text{speculative tilt}}\; }
\label{eq:markowitz}
\end{equation}
i.e.\ the global minimum-variance portfolio plus a return-seeking tilt whose size scales with
risk tolerance $1/\lambda$. Equation~\eqref{eq:markowitz} is the standard two-fund result of
mean--variance analysis: every optimal portfolio is the same fixed combination of the
minimum-variance portfolio and a single return-seeking direction, mixed in a proportion set by
$\lambda$.\footnote{The signs of the two terms are the easiest place to slip when the multiplier is
eliminated: reversing the sign of the tilt (or of the $\mathsf b/\mathsf a$ correction inside it)
yields a formula that appears dimensionally sensible but sends the portfolio \emph{away} from the
minimum-variance point as risk aversion rises. The signs in~\eqref{eq:markowitz} are pinned down by
the limit $\lambda\to\infty$, in which the tilt must vanish and $B$ must reduce to the global
minimum-variance portfolio $\Sig^{-1}\mathbf 1/\mathsf a$; \eqref{eq:markowitz} satisfies this
check, and we verify it numerically in Section~\ref{sec:results}.} We take $B$ as the destination
of every reallocation studied below.

\subsection{A risk metric on the simplex}\label{sec:metric}

The heart of the framework is a Riemannian metric $g(w)$: a smoothly varying, symmetric,
positive-definite bilinear form on each tangent space, encoding the instantaneous risk of trading
in a given direction from a given allocation. We build it additively from the two risks priced by
the model.

\emph{Market risk} is the return covariance $\Sig$. Moving the book by an infinitesimal trade
$\mathrm dw$ changes portfolio return variance at second order by $\mathrm dw\T\Sig\,\mathrm dw$;
this is the natural, and constant, market contribution.

\emph{Concentration risk} penalises exposure to individual names. We use the Fisher--Rao
(Shahshahani) metric of the simplex, $\diag(1/w_1,\dots,1/w_n)$, which is positive-definite on the
interior and diverges as any weight approaches zero---so drifting toward a vertex (a single-name
portfolio) becomes geometrically expensive. This choice is principled (it is \emph{the} information
metric of the simplex), guarantees a well-posed problem, and connects directly to the information-
geometry literature reviewed above.

Combining them with a concentration-aversion weight $\kappa\ge0$ gives the \emph{market\,+\,concentration
metric}
\begin{equation}
G(w)=\Sig+\kappa\,\diag\!\Big(\tfrac1{w_1},\dots,\tfrac1{w_n}\Big)\ \in\R^{n\times n},
\qquad \kappa\ge0 .
\label{eq:fisher-metric}
\end{equation}
Because $\Sig\succ0$ and $\diag(1/w)\succ0$ on the interior, $G(w)\succ0$ throughout, and the
induced metric in the chart $u$ is
\begin{equation}
g(u)=J\T\,G\big(w(u)\big)\,J\ \in\R^{(n-1)\times(n-1)},\qquad g(u)\succ0 .
\label{eq:induced}
\end{equation}
Setting $\kappa=0$ recovers the pure market metric $G\equiv\Sig$, constant in $w$.

\subsection{Transition risk as length; geodesics}\label{sec:length}

A reallocation is a curve $\gamma:[0,1]\to\Simplex$ with $\gamma(0)=A$, $\gamma(1)=B$. Its
\emph{transition risk} is the Riemannian length
\begin{equation}
\risk[\gamma]=\int_0^1\sqrt{\dot\gamma(t)\T\,G\big(\gamma(t)\big)\,\dot\gamma(t)}\;\mathrm dt ,
\label{eq:risklen}
\end{equation}
the accumulated ``risk travelled'' as the book moves from $A$ to $B$. Equation~\eqref{eq:risklen}
is reparametrisation-invariant: it measures the geometric route, not the speed schedule. The
least-risk reallocation is its minimiser,
\begin{equation}
\gamma^\star=\argmin_{\gamma(0)=A,\ \gamma(1)=B}\ \risk[\gamma],
\end{equation}
which is a \emph{geodesic} of $G$. Minimisers of $\risk$ coincide with minimisers of the energy
$E[\gamma]=\int_0^1\dot\gamma\T G\dot\gamma\,\mathrm dt$ (by Cauchy--Schwarz, with equality for
affine-parametrised geodesics), and the latter is numerically better behaved; we minimise $E$ and
report $\risk$. In coordinates the geodesic satisfies
\begin{equation}
\ddot u^{\,k}+\Gamma^k_{ij}(u)\,\dot u^{\,i}\dot u^{\,j}=0,\qquad u(0)=u_A,\ u(1)=u_B,
\label{eq:geodesic}
\end{equation}
with Christoffel symbols
$\Gamma^k_{ij}=\tfrac12 g^{k\ell}\big(\partial_i g_{j\ell}+\partial_j g_{i\ell}-\partial_\ell g_{ij}\big)$
built from the induced metric~\eqref{eq:induced}.

\subsection{Three reallocations from \texorpdfstring{$A$}{A} to \texorpdfstring{$B$}{B}}\label{sec:paths}

We contrast three trajectories with identical endpoints, written $p_0,p_1,p_2$ throughout.

\smallskip\noindent\textbf{(i) Direct rebalancing --- the straight line $p_0$.}
\begin{equation}
p_0(t)=(1-t)\,A+t\,B .
\end{equation}
This is the shortest Euclidean route and, as Proposition~\ref{prop:flat} shows, the exact
risk-geodesic of the market-only metric.

\smallskip\noindent\textbf{(ii) Myopic rebalancing --- the gradient flow $p_1$.}
The greedy strategy always trades in the direction that most improves the Markowitz objective
$f(w)=\mu\T w-\lambda w\T\Sig w$, subject to staying on the budget hyperplane. With the
projection $P=I-\tfrac1n\mathbf 1\mathbf 1\T$ onto $T_w\Simplex$, this is the projected
gradient-ascent flow
\begin{equation}
\dot w = P\,\nabla f(w) = P\big(\mu-2\lambda\Sig w\big),\qquad w(0)=A .
\label{eq:gradflow}
\end{equation}
Since $\mathbf 1\T P=0$ the budget is conserved, and the unique stationary point of
\eqref{eq:gradflow} on the hyperplane is precisely the Markowitz target $B$ of
\eqref{eq:markowitz}: the flow \emph{does} arrive at $B$. It embodies ``improve the portfolio as
fast as possible at each instant'' while completely ignoring the risk of the transition itself.

\smallskip\noindent\textbf{(iii) Risk-optimal rebalancing --- the geodesic $p_2$.}
The minimiser $p_2=\gamma^\star$ of \eqref{eq:risklen} under the full metric $G$, solving the
boundary-value problem \eqref{eq:geodesic}. Unlike $p_1$, it accounts for the cumulative risk of
\emph{every} intermediate portfolio, and unlike the line it is sensitive to how risk varies across
the simplex.

\subsection{Curvature}\label{sec:curvature-def}

Whether a geodesic can differ from a straight line is governed by curvature. For the two-dimensional
case ($n=3$) we compute the Gaussian curvature
\begin{equation}
K=\frac{R_{1212}}{\det g},\qquad
R^{\ell}_{\ i j k}=\partial_i\Gamma^{\ell}_{jk}-\partial_j\Gamma^{\ell}_{ik}
+\Gamma^{\ell}_{im}\Gamma^{m}_{jk}-\Gamma^{\ell}_{jm}\Gamma^{m}_{ik},
\label{eq:curv}
\end{equation}
from the induced metric~\eqref{eq:induced} by finite differences. Curvature is the intrinsic,
coordinate-free signal that the risk landscape is genuinely warped; $K\equiv0$ means it is flat and
straight-line rebalancing cannot be improved upon.

\section{Theoretical results}\label{sec:theory}

\subsection{When geometry is redundant: flatness of the market metric}

\begin{proposition}[Market risk alone $\Rightarrow$ direct rebalancing is optimal]\label{prop:flat}
If only market risk is priced ($\kappa=0$, so $G\equiv\Sig$ is constant), then all Christoffel
symbols vanish, the manifold is flat ($K\equiv0$), and the unique risk-geodesic between any two
portfolios $A,B$ is the straight-line segment $p_0(t)=(1-t)A+tB$. Consequently the
direct rebalancing path is exactly risk-optimal, and no curved route can reduce transition risk.
\end{proposition}

\begin{proof}
With $G$ constant the induced metric \eqref{eq:induced} $g=J\T\Sig J$ is constant in $u$, so
$\partial_i g_{j\ell}\equiv0$ and hence $\Gamma^k_{ij}\equiv0$. The geodesic
equation~\eqref{eq:geodesic} reduces to $\ddot u=0$, whose solutions with the given endpoints are
the affine segments $u(t)=(1-t)u_A+tu_B$, i.e.\ straight lines in $w$. Every term of the Riemann
tensor~\eqref{eq:curv} is built from $\Gamma$ and its derivatives, so $R\equiv0$ and $K\equiv0$.
A constant positive-definite quadratic form makes $(\Simplex,g)$ isometric to a convex subset of
Euclidean space, in which the segment is the unique length-minimiser.
\end{proof}

Proposition~\ref{prop:flat} is the framework's honesty clause: it identifies exactly when the
geometric machinery is unnecessary. If a manager prices nothing but covariance risk, they should
rebalance in a straight line, and this paper has nothing to add. Everything interesting happens
when the metric varies with position.

\subsection{When geometry helps, it never hurts: weak domination}

\begin{proposition}[The geodesic weakly dominates every path]\label{prop:dom}
Let $G$ be any positive-definite risk metric and let $\gamma^\star$ be a risk-length--minimising
geodesic from $A$ to $B$. Then for every admissible reallocation $\gamma$ with the same endpoints,
$\risk[\gamma^\star]\le\risk[\gamma]$. In particular
\begin{equation}
\risk[p_2]\ \le\ \risk[p_0]
\qquad\text{and}\qquad
\risk[p_2]\ \le\ \risk[p_1],
\end{equation}
so the transition-risk saving of the geodesic over both direct and myopic rebalancing is always
non-negative. The improvement is strictly positive whenever the competing path is not itself a
geodesic of $G$.
\end{proposition}

\begin{proof}
Immediate from the definition of $\gamma^\star$ as the minimiser of $\risk$ over all curves joining
$A$ to $B$; the line and the gradient flow~\eqref{eq:gradflow} are two such curves. Strictness
follows because a non-geodesic curve has strictly larger length than the minimising geodesic in a
Riemannian manifold.
\end{proof}

Proposition~\ref{prop:dom} is the core value proposition. Unlike heuristics that may help on
average but backfire on individual transitions, the geodesic reallocation carries a guarantee: it
can never do worse than rebalancing directly. The only questions are \emph{how much} it helps and
\emph{when}---questions we answer computationally.

\subsection{Where large savings come from: endogenous, crowding-driven risk}\label{sec:endo}

Why is the concentration saving of metric~\eqref{eq:fisher-metric} typically modest? Because the
straight line already diversifies. Portfolio concentration is naturally measured by the Herfindahl
index $\Herf(w)=\sum_i w_i^2$, a \emph{convex} function of $w$. Along the chord,
$\Herf\big((1-t)A+tB\big)\le(1-t)\Herf(A)+t\Herf(B)$, so the midpoint of a direct rebalancing is
\emph{less} concentrated than the endpoints. Any purely concentration-based penalty is thus already
small in the interior of the chord, leaving little for a geodesic to improve. This is not a defect
of the computation: it is a theorem-level reason why direct rebalancing is already good against
concentration risk, and it explains the small, strictly positive savings we report for
metric~\eqref{eq:fisher-metric}.

Large savings require a priced risk that is \emph{not} convex-favourable to the chord. Real markets
supply one. As the model's own motivation notes, correlations are not static: crowded co-holdings
of correlated names co-crash in stress (``correlation breakdown''), so \emph{holding two correlated
assets together} is far riskier than the static covariance suggests. Crucially, the exposure of a
pair $(i,j)$ to this effect scales like the \emph{product} $w_iw_j$---a \emph{bilinear}, non-convex
quantity. For a transition from a portfolio concentrated in asset $i$ to one concentrated in asset
$j$, the endpoints have $w_iw_j\approx0$, but the midpoint of the direct chord has
$w_iw_j\approx\tfrac14$: the straight line drives the book straight into the crowded region, even
though both endpoints avoid it.

We encode this with a scalar \emph{instability field} over a set $\mathcal C$ of crowded (highly
correlated) pairs,
\begin{equation}
\Phi(w)=1+\eta\!\!\sum_{(i,j)\in\mathcal C}\! w_i w_j,\qquad \eta\ge0,
\label{eq:phi}
\end{equation}
and define the \emph{endogenous risk metric}
\begin{equation}
G_{\mathrm{endo}}(w)=\Phi(w)\,\Big(\Sig+\kappa\,\diag(1/w)\Big).
\label{eq:endo-metric}
\end{equation}
Moving \emph{anywhere} while the book sits in a crowded configuration is amplified by $\Phi$, so a
geodesic is repelled from the high-$\Phi$ ridge and, if a third uncorrelated asset is available,
routes through it: it \emph{unwinds} the position in $i$ into the safe asset first and only then
builds $j$, never holding the crowded pair together. This is the ``intentionally take a detour to
avoid a dangerous intermediate state'' behaviour that direct rebalancing cannot express, and---as
Section~\ref{sec:results} shows---it produces transition-risk savings an order of magnitude larger
than the concentration metric. Metric~\eqref{eq:endo-metric} still reduces to the flat market metric
as $\eta,\kappa\to0$, so Propositions~\ref{prop:flat} and~\ref{prop:dom} continue to apply.

\section{Computational methodology}\label{sec:methodology}

\paragraph{Market.} To keep the study reproducible and free of vendor data we simulate a
six-asset market from a two-factor return model (a broad ``market'' factor and a ``growth''
factor plus idiosyncratic noise) over $N=1000$ daily observations, and estimate $\mu$ and $\Sig$
from the simulated panel. The estimated annualised volatilities span $13.8\%$--$28.9\%$ and
pairwise correlations $0.30$--$0.46$, a realistic moderately-correlated equity cross-section.
Because sample means over $10^3$ days are notoriously noisy---the very sensitivity that makes raw
Markowitz portfolios extreme---we use the model's expected returns as the return view, playing the
role of an equilibrium or Black--Litterman input; this keeps the target $B$ interior and the
demonstration clean. At $\lambda=4$ the target~\eqref{eq:markowitz} is the well-diversified
$B=(0.15,0.14,0.08,0.27,0.10,0.27)$, and we take a deliberately concentrated start
$A=(0.55,0.05,0.08,0.05,0.22,0.05)$.

\paragraph{Metrics.} We set the concentration weight to $\kappa=4\kappa_0$ with
$\kappa_0=\bar{\Sig}_{ii}/n$ the natural scale that makes the Fisher term commensurate with the
market term at the centroid. For the endogenous metric~\eqref{eq:endo-metric} the crowded set
$\mathcal C$ is the two most-correlated asset pairs (correlations $0.46$ and $0.44$) and
$\eta=25$.

\paragraph{Geodesic solver.} We compute geodesics by direct minimisation of the discretised energy
$E[\gamma]=K\sum_{k=0}^{K-1}\Delta w_k\T\,G(\bar w_k)\,\Delta w_k$ over $K\!+\!1$ nodes with fixed
endpoints, where $\bar w_k$ is the midpoint of segment $k$. Each interior node is parametrised as
$w_k=\mathrm{softmax}(z_k)$, which keeps every node \emph{strictly inside} the simplex regardless
of the optimiser's steps, so the Fisher term never blows up. The energy and its exact gradient (by
automatic differentiation) are handed to a quasi-Newton (L-BFGS) minimiser, warm-started from the
straight line. This construction respects Proposition~\ref{prop:dom} by design: the straight line
is a feasible point, so the returned geodesic never has larger energy. The gradient-flow path
$p_1$ is obtained by integrating~\eqref{eq:gradflow} to convergence and reporting its
(reparametrisation-invariant) risk-length. Gaussian curvature~\eqref{eq:curv} is evaluated by
finite differences of the induced metric.

\section{Results}\label{sec:results}

\subsection{Setup: the efficient frontier and the transition}

Figure~\ref{fig:frontier} shows the efficient frontier of the six-asset market with the
concentrated start $A$ and the mean--variance target $B$. $A$ sits well inside the frontier (it is
both riskier and lower-returning than efficient portfolios of comparable volatility); the task is
to travel from $A$ to $B$ at least accumulated risk.

\begin{figure}[t]
\centering
\includegraphics[width=0.68\linewidth]{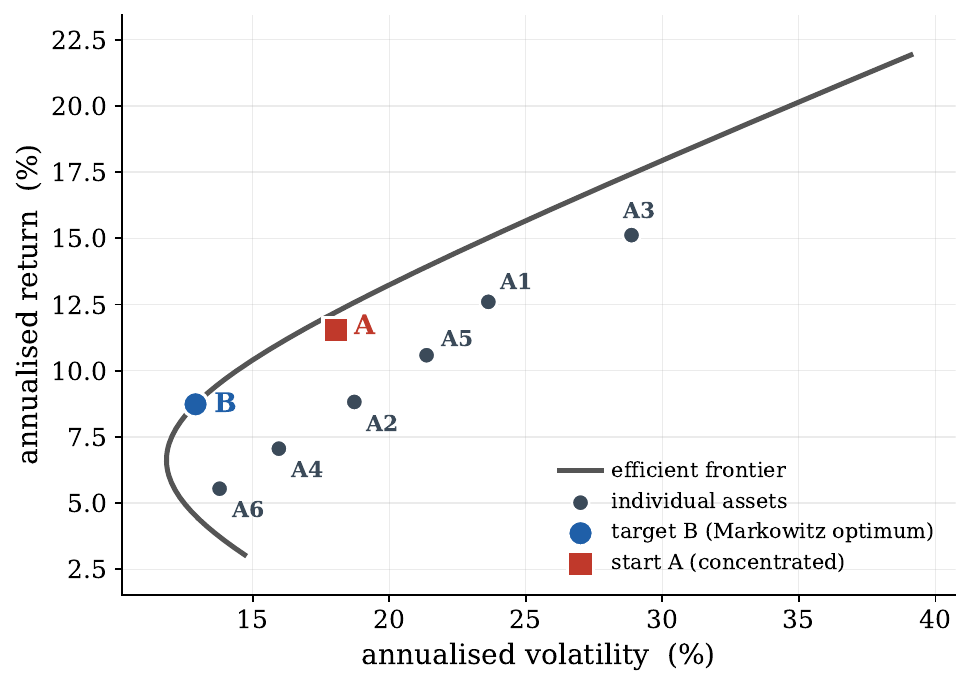}
\caption{Efficient frontier of the simulated six-asset market. The concentrated starting portfolio
$A$ (red square) is interior and dominated; the Markowitz target $B$ (blue) lies on the frontier.
The labelled points $A_1,\dots,A_6$ are the individual assets. Every reallocation studied connects
these two points.}
\label{fig:frontier}
\end{figure}

\subsection{The mechanism, in three assets}\label{sec:tri}

The clearest view of \emph{why} a geodesic helps is in three assets, where the simplex is a
triangle we can draw (Figure~\ref{fig:triangle}). Assets $0$ and $1$ are a correlated, crowded
pair; asset $2$ is a relatively uncorrelated ``diversifier''. The book starts concentrated in
asset $0$ ($A$) and must move to a portfolio concentrated in asset $1$ ($B$). The background shading
is the instability field $\Phi(w)$ of~\eqref{eq:phi}: it is highest (red) along the bottom edge,
where the crowded pair is co-held, and lowest (green) toward the asset-$2$ vertex.

The straight line (red dashed) runs straight along the bottom edge---through the most dangerous
region---because linear interpolation between a portfolio in asset $0$ and one in asset $1$ holds
both together at every intermediate step. The risk geodesic (blue) instead \emph{bows upward
toward the diversifier}: it unwinds asset $0$ into asset $2$, carries the exposure through the
safe region, and only then builds asset $1$, temporarily allocating up to $18\%$ to asset $2$ (the
straight line never exceeds $5\%$). It thereby avoids ever holding the crowded pair at full weight.
This detour is not free in Euclidean terms---it is a longer trade list---but in risk terms it is
$4.3\%$ cheaper. The mechanism is exactly the ``navigate around the high-risk area'' picture, now
realised by an actual solved geodesic rather than a schematic.

\begin{figure}[t]
\centering
\includegraphics[width=0.66\linewidth]{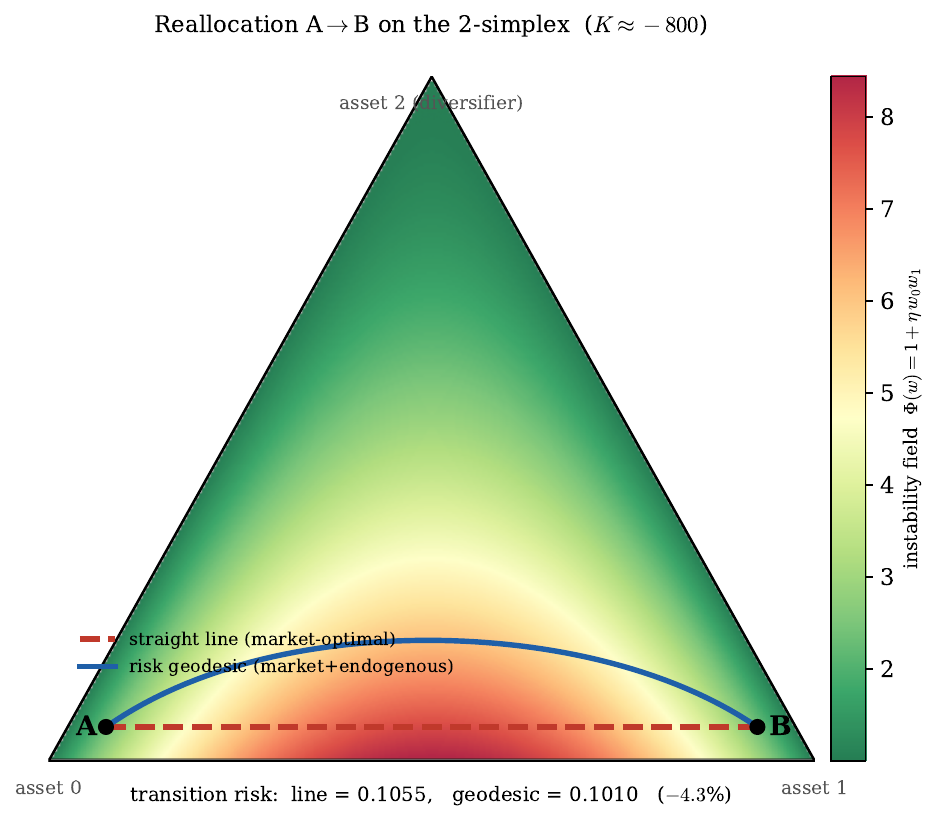}
\caption{A reallocation on the three-asset simplex under the endogenous metric. Background: the
instability field $\Phi$ (red = crowded/high-risk, green = safe). The straight line
(market-optimal) ploughs along the crowded bottom edge; the risk geodesic arcs toward the
uncorrelated asset-$2$ vertex, cutting transition risk by $4.3\%$. The manifold has strongly
negative Gaussian curvature ($K\approx-8.0\times10^2$), the intrinsic signature of the warped risk
landscape.}
\label{fig:triangle}
\end{figure}

\paragraph{Curvature confirms the picture.} We compute the Gaussian curvature of this three-asset
manifold at its centre for all three metrics. The market-only metric gives $K=0.000$ to numerical
precision---flat, exactly as Proposition~\ref{prop:flat} demands, so its geodesic is the straight
line. Adding the Fisher concentration term gives $K\approx-4.3\times10^2$, and the full endogenous
metric $K\approx-8.0\times10^2$. Curvature is strictly negative and grows in magnitude as more
position-dependent risk is priced: the landscape is genuinely, and increasingly, warped, and
geodesics must bend.

\subsection{The transition in six assets}

Figure~\ref{fig:weights} shows the full six-asset reallocation under the endogenous metric: each
asset's weight along the straight line (dashed) and along the geodesic (solid). The geodesic keeps
the weights of the crowded names lower during the middle of the transition and leans on the less
correlated names to carry exposure across, before converging to the same target $B$. The
qualitative signature of the three-asset case survives in higher dimension, though it is visually
subtler.

\begin{figure}[t]
\centering
\includegraphics[width=0.74\linewidth]{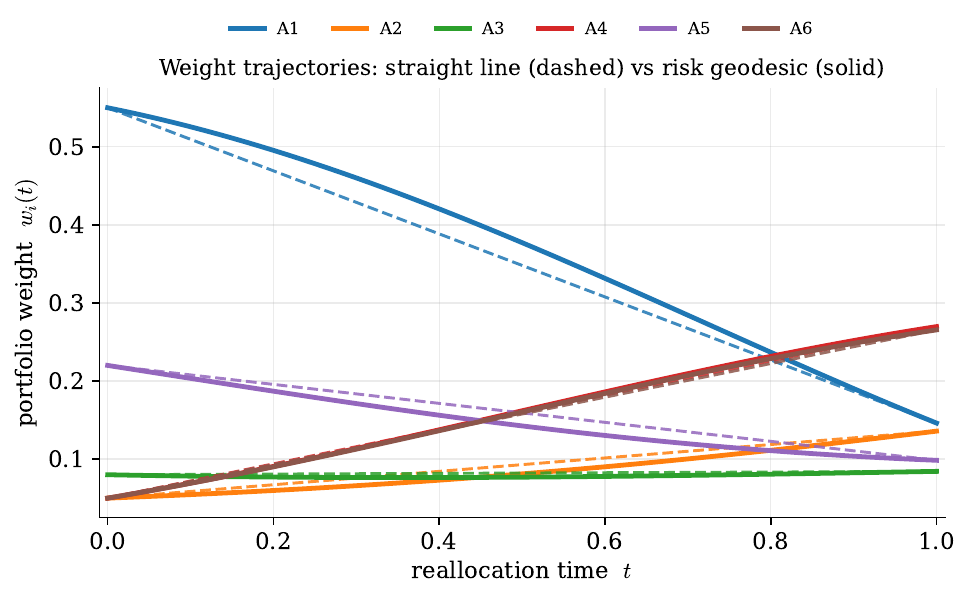}
\caption{Weight trajectories for the six-asset reallocation $A\to B$ under the endogenous metric:
straight line (dashed) versus risk geodesic (solid), per asset. Both paths share the endpoints;
the geodesic redistributes the intermediate exposures to avoid crowded co-holdings.}
\label{fig:weights}
\end{figure}

For this single transition the geodesic beats the straight line by $0.98\%$ under the endogenous
metric and by $0.25\%$ under the concentration metric, and it beats the \emph{myopic gradient-flow}
path $p_1$ by $3.9\%$: the greedy strategy, which improves the Markowitz objective as fast as
possible, is markedly risk-inefficient during the transition because it drives hard through
whatever region is locally most return-improving, heedless of the risk it accumulates on the way.

\subsection{How savings scale with priced risk}

Proposition~\ref{prop:flat} says the saving must vanish when position-dependent risk is switched
off. Figure~\ref{fig:scaling} confirms this and shows how the saving grows as it is switched on.
Under the concentration metric the reduction rises monotonically from zero with $\kappa$,
saturating below $0.30\%$ for this transition: concentration risk alone buys little, for the
convexity reason of Section~\ref{sec:endo}. Under the endogenous metric the reduction grows
steadily with the crowding strength $\eta$, reaching $2.65\%$ at $\eta=50$---an order of magnitude
larger, and still climbing (at $\eta=0$ it coincides with the small concentration-metric value,
since the endogenous metric then reduces to the Fisher metric). The contrast between the two panels
is the empirical face of the convex-versus-bilinear distinction.

\begin{figure}[t]
\centering
\includegraphics[width=0.92\linewidth]{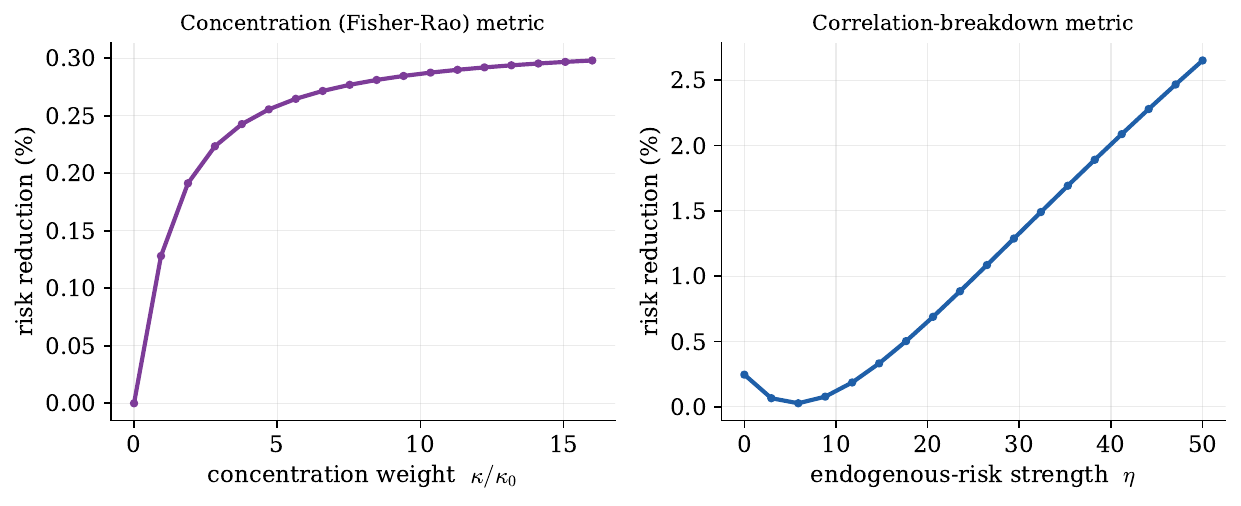}
\caption{Transition-risk reduction (geodesic vs.\ line) as a function of priced position-dependent
risk. Left: concentration weight $\kappa$ (Fisher--Rao); the effect is real, monotone, but modest.
Right: crowding strength $\eta$ (endogenous metric, concentration weight held fixed); the effect is
an order of magnitude larger and still growing. The left curve starts at zero when $\kappa=0$ (the
pure market metric, Proposition~\ref{prop:flat}); the right curve starts from the small
concentration-metric baseline and climbs steeply with $\eta$.}
\label{fig:scaling}
\end{figure}

\subsection{Monte-Carlo distribution of savings}

To characterise the saving beyond a single transition we draw $100$ random starting portfolios
$A\sim\mathrm{Dirichlet}(0.6\cdot\mathbf 1)$---spanning diffuse and concentrated books---and
reallocate each to the common target $B$. Figure~\ref{fig:hist} shows the distribution of the
risk saving of the geodesic over the straight line, under both metrics.

\begin{figure}[t]
\centering
\includegraphics[width=0.74\linewidth]{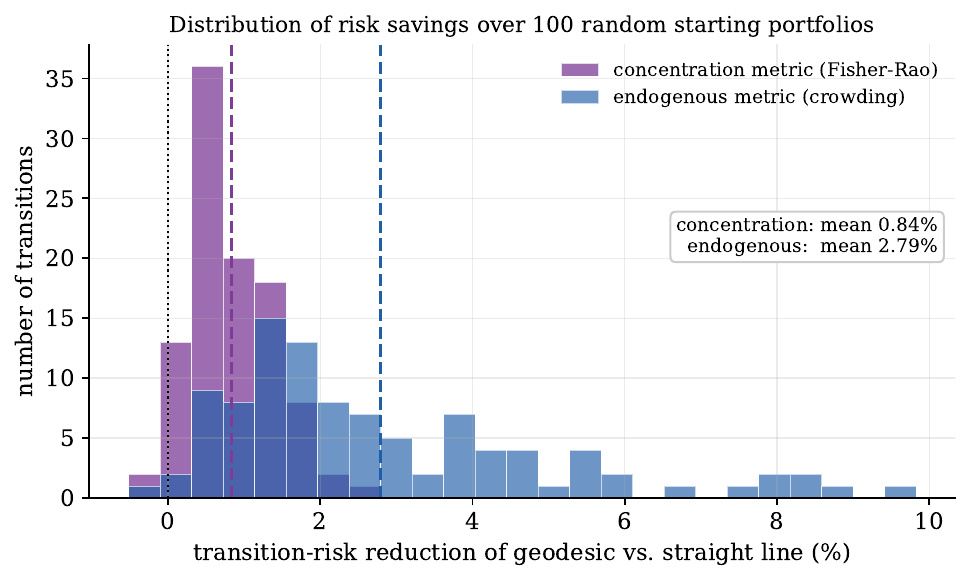}
\caption{Distribution of the transition-risk saving (geodesic vs.\ straight line) over $100$
random starting portfolios reallocated to the same target. Dashed lines mark the means. The
concentration metric saves $0.84\%$ on average; the endogenous metric $2.79\%$, with a right tail
beyond $9\%$. Both distributions sit almost entirely at or above zero, as the domination theorem
requires.}
\label{fig:hist}
\end{figure}

The concentration metric yields a mean saving of $0.84\%$ (median $0.72\%$, s.d.\ $0.53\%$,
maximum $2.47\%$); the endogenous metric a mean of $2.79\%$ (median $2.18\%$, s.d.\ $2.16\%$,
maximum $9.53\%$). Consistent with Proposition~\ref{prop:dom}, $98\%$ and $99\%$ of transitions
respectively have a non-negative saving; the handful of tiny negative values (all below $0.5\%$ in
magnitude) reflect finite solver tolerance rather than a genuine violation, since the straight
line is always available to the minimiser. The endogenous saving is not a knife-edge phenomenon:
it is present across the ensemble and routinely reaches several percent.

\subsection{What predicts the saving?}\label{sec:regression}

Which transitions benefit most? We regress the endogenous saving on three geometric features of
the transition: its Euclidean size $\lVert A-B\rVert$; the \emph{excess crowding} of the direct
chord, defined as the peak co-holding $\sum_{(i,j)\in\mathcal C}w_iw_j$ along the straight line
minus the larger of its endpoint values (i.e.\ how far the chord is forced into the crowded region
\emph{beyond} what the endpoints require); and the starting concentration $\Herf(A)$.
Table~\ref{tab:reg} reports the fit.

\begin{table}[t]
\centering
\caption{OLS regression of the endogenous transition-risk saving (in \%) on the geometry of the
transition, over $N=100$ random reallocations to the common target ($R^2=0.51$). Positive
coefficients mean the feature increases the saving.}
\label{tab:reg}
\setlength{\tabcolsep}{10pt}
\renewcommand{\arraystretch}{1.25}
\begin{tabular}{@{}l r r r r@{}}
\toprule
Feature & Coefficient & Std.\ error & $t$-statistic & $p$-value \\
\midrule
Intercept                            & $0.72$  & $0.58$ & $1.25$  & $0.22$ \\
Transition size $\lVert A-B\rVert$   & $7.89$  & $2.03$ & $3.89$  & $1.9\times10^{-4}$ \\
Excess chord crowding                & $265.3$ & $40.4$ & $6.57$  & $2.6\times10^{-9}$ \\
Start concentration $\Herf(A)$       & $-5.03$ & $2.64$ & $-1.91$ & $0.059$ \\
\bottomrule
\end{tabular}
\end{table}

Two effects stand out. Larger transitions save more ($t=3.9$): there is simply more risky
distance over which to optimise the route. But the dominant predictor by far is \emph{excess
crowding} ($t=6.6$, $p<10^{-8}$): the more the direct chord is forced through the non-convex
crowded region relative to its endpoints, the more the geodesic saves by routing around it. This
is precisely the mechanism of Section~\ref{sec:endo}, now confirmed statistically---the saving is
a function of \emph{how badly the straight line misbehaves}, which is exactly what a bilinear,
non-convex risk creates and a convex one does not. The starting concentration $\Herf(A)$ adds
little once these are controlled for.

\paragraph{The centroid as a ``fixed point''.} The regression also formalises an intuition about
where geometry is inert. When $A$, $B$ and the most-diversified portfolio (the simplex centroid,
the global minimum of concentration risk) are close to collinear, the direct chord already tracks
the low-risk corridor, excess crowding is near zero, and the geodesic barely deviates from the
line. The centroid thus behaves like an attractor of ``easy'' transitions: reallocations aimed
along the diversified direction gain little from the geometry, while reallocations that would
otherwise cut across crowded regions gain the most. This sharpens the original conjecture that a
special low-risk point governs the size of the deviation.

\section{Discussion}\label{sec:discussion}

\paragraph{An honest account of magnitude.} The framework does not manufacture large numbers. When
the priced risk is concentration---a convex quantity---the straight line is already close to
optimal because it diversifies at its midpoint, and the geodesic saves only tens of basis points.
We regard this not as a weakness but as a feature: it is a theorem-backed statement about
\emph{when} direct rebalancing is good enough, and the small savings we report ($20$--$80$ bps)
are exactly the honest magnitude such transitions admit. Yet tens of basis points of transition
risk, saved on every rebalancing of a large book and compounded over many rebalancings, is not
economically negligible; and the framework tells the manager, in advance and per transition,
whether the number will be small or large.

\paragraph{Where the geometry earns its keep.} The saving becomes first-order---several percent,
with tails approaching $10\%$---exactly when the priced risk is \emph{not} convex-favourable to
the chord. Endogenous, crowding-driven risk is the canonical case: because exposure to correlation
breakdown scales bilinearly in the weights, the direct path is forced into the danger zone that a
geodesic can circumvent. This is the regime where ``take a temporary detour to avoid a dangerous
intermediate state'' is a genuinely non-obvious, genuinely valuable strategy, and it is precisely
the regime a manager cannot navigate by intuition alone.

\paragraph{Curvature as the diagnostic.} Gaussian curvature cleanly separates the two regimes:
$K\equiv0$ for the market metric (geometry redundant), $K<0$ and growing once position-dependent
risk is priced (geometry active). This operationalises the literature's association of curvature
with concentration and systemic fragility \cite{sandhu2016}: here curvature is not merely a
correlate of risk but the computable quantity that certifies whether a curved reallocation can beat
a straight one.

\section{Limitations and extensions}\label{sec:limitations}

\emph{Transaction costs.} A quadratic market-impact cost of trading adds a constant
positive-definite term $\Lambda$ to the metric, $G\mapsto G+\Lambda$; being constant it is flat and
does not by itself bend geodesics, but it does reweight the trade-off and is trivially
accommodated. Linear (proportional) costs break the Riemannian (quadratic) structure and lead to a
Finsler metric---a natural but more demanding extension.

\emph{Richer endogenous risk.} We modelled crowding with a static instability field; a fuller
treatment would let the effective covariance itself depend on the weights,
$\Sig\mapsto\Sig_{\mathrm{eff}}(w)$, capturing liquidity-driven volatility inflation of large
single-name positions and regime-dependent correlations. The machinery is unchanged---any smooth
SPD field is a metric---but the calibration is an empirical project in its own right.

\emph{Estimation and robustness.} The metric inherits the estimation error of $\Sig$ and $\mu$;
pairing the geometric route with shrinkage or distributionally-robust covariance estimates
\cite{blanchet2022} would guard the path against misspecification, and is a promising direction.

\emph{Scale and dynamics.} The energy-minimisation solver scales to realistic asset counts; the
open frontier is the \emph{dynamic} problem, in which $B$ itself drifts as parameters update, so
that the target is moving and the geodesic must be recomputed on the fly---a tracking problem on a
time-varying manifold.

\section{Conclusion}\label{sec:conclusion}

Mean--variance optimisation answers \emph{where} a portfolio should go; it is silent on \emph{how}
to get there. We have framed that second question as a problem of Riemannian geometry: allocations
are points of the simplex, risk defines a metric, the accumulated risk of a rebalancing is the
length of a curve, and the least-risk transition is a geodesic. The picture yields two clean
guarantees---direct rebalancing is exactly optimal when only market risk is priced, and the
geodesic never underperforms the straight line or the myopic flow when it is not---and a sharp,
computable criterion, curvature, for telling the two regimes apart. Applied to concentration risk
the savings are modest and honest; applied to endogenous, crowding-driven risk---where the direct
path is forced through a danger zone that a geodesic can route around---they become economically
significant, reaching several percent, and are predicted by a single geometric feature of the
transition. The result is a general-purpose, quantitative tool for transition management: the
manager supplies a starting book, a target, and the risks they wish to price, and the geometry
returns the smoothest, least-risky route between them.

\end{document}